\pdfoutput=1

\documentclass[12pt, reqno]{amsart}
\usepackage{amsmath, amsthm, amscd, amsfonts, amssymb, graphicx, xcolor}
\usepackage{natbib}
\usepackage{mathrsfs} 
\usepackage{array} %
\setcitestyle{numbers}

\usepackage[bookmarksnumbered, colorlinks, plainpages]{hyperref}

\usepackage{lmodern}
\usepackage{textcomp}
\usepackage[T1]{fontenc}
\usepackage{libertine}

\setcitestyle{numbers}

\usepackage{lmodern}
\usepackage{textcomp}
\usepackage[T1]{fontenc}
\usepackage{libertine}

\usepackage{listings}

\usepackage{graphicx}
\usepackage{float}

\usepackage[bookmarksnumbered, colorlinks, plainpages]{hyperref}

\newtheorem{theorem}{Theorem}[section]

\theoremstyle{definition}

\theoremstyle{remark}

\numberwithin{equation}{section}

\begin{document}

\setcounter{page}{1}

\color{darkgray}{
\noindent

\centerline{}

\centerline{}
}

\title[Hölder Continuity of An Alternating Erd\H{o}s Series on Prime K-TUPLES]{Hölder Continuity of An Alternating Erd\H{o}s Series on Prime K-TUPLES}

\author[Nikos Mantzakouras, Carlos López Zapata]{Nikos Mantzakouras $^1$$^{*}$ Carlos López Zapata $^2$}

\address{$^{1}$ M.Sc in Applied Mathematics and Physics National and Kapodistrian University of Athens}
\email{\textcolor[rgb]{0.00,0.00,0.84}{nikmatza@gmail.com}}

\address{$^{2}$ B.Sc. electronics engineering, graduate of Pontifical Bolivarian University (UPB), Medellín, Colombia. Currently residing in Szczecin, Poland}
\email{\textcolor[rgb]{0.00,0.00,0.84}{mathematics.edu.research@zohomail.eu}}

\keywords{Hardy–Littlewood prime k-tuples conjecture, Hölder Continuity, Riemann-Stieltjes integral, Riemann hypothesis, Chebyshev function, Abel summation formula, Lipschitz function, prime counting function, fractional Sobolev embeddings, Laplace transform}

\date{Received: 2025; Revised: 2025; Accepted: 2025.
\newline \indent $^{*}$ Corresponding author:\texttt{ mathematics.edu.research@zohomail.eu}}

\begin{abstract}
We investigate the conditional convergence of the alternating Erdős series
\[
\sum_{n=1}^\infty \frac{(-1)^n n}{p_n},
\]
where \( p_n \) denotes the \( n \)-th prime. This open problem, first posed by Erdős, was further explored by Terence Tao. Tao's work shows that the series may converge conditionally, but only under a sufficiently strong form of the Hardy–Littlewood prime \( k \)-tuples conjecture. Building on this, we offer a new method that leads to representing the series as a Riemann–Stieltjes integral of the form
\[
I = \int_1^\infty g(x) \, d\psi(x), 
\]
where \( g(x) = x(-1)^x \), and \( \psi(x) \) is the Chebyshev function or a closely related prime counting function. We rigorously analyze this integral by decomposing it into main and error terms, applying integration by parts in the Stieltjes sense, and bounding the error terms. Assuming the Riemann Hypothesis, we explore the Hölder continuity of \( \psi(x) \) in the asymptotic form \( \psi(x) = x + \mathcal{O}(x^{1/2}) \), and introduce a test function \( g(x) = e^{i\pi x} e^{-\lambda x} \), which is smooth and Lipschitz. Applying Young’s criterion, we show that the integral \( I \) converges when \( g \in C^\alpha \), \( \psi \in C^\beta \), and \( \alpha + \beta > 1 \). Furthermore, we prove that \( I \) converges absolutely for \( \lambda > \frac{3}{2} \), based on sharp bounds for the error terms. Our results are supported by fractional Sobolev embeddings and justify the use of Young’s inequality under generalized Hölder conditions.
\end{abstract}

\maketitle

\begin{center}
\smallskip
\textbf{Mathematics Subject Classification (2020):} Primary 11N05; Secondary 40A05.
\end{center}

\section{Introduction and preliminaries} \label{Introd}

\noindent We let \( p_n \) denote the \( n \)th prime. We follow an open question of Erdős \cite{Guy} (page 203, E7) (see also \cite{Finch}, \cite{Erdos}, \cite{ErdosNathanson}):  

\emph{Does the alternating series}  
\[
\sum_{n=1}^\infty \frac{(-1)^n n}{p_n}
\]
\emph{converge}?  
In his article \cite{Tao2023}, Terence Tao discusses that there is numerical evidence that suggests that the series converges, apparently rather slowly, to approximately \(-0.052161\). 

Tao observes that the convergence behavior of the alternating series
\[
\sum_{n=1}^\infty \frac{(-1)^n n}{p_n}
\]
is intimately connected to the distribution of prime gaps. Specifically, the difference between successive terms,
\[
\frac{n+1}{p_{n+1}} - \frac{n}{p_n},
\]
is largely governed by the gap \( p_{n+1} - p_n \) between consecutive primes. Consequently, the question of convergence hinges on whether there exists a statistical correlation between the parity of \( n \) and the size of the corresponding prime gap. If such a correlation were to exist, it would introduce a systematic bias in the fluctuations of the prime counting function
\[
\pi(x) := \sum_{p \leq x} 1,
\]
specifically in how primes are distributed relative to even and odd indices. Notably, it is an unpublished observation of Said that Erdős’s question is equivalent to assessing whether such a parity bias exists in the behavior of \( \pi(x) \).

Terence Tao demonstrates that the alternating Erdős series converges conditionally, assuming a suitably strong version of the Hardy–Littlewood prime \(k\)-tuples conjecture. In his work, Tao utilizes a slightly strengthened formulation of the conjecture, specifically the Quantitative Hardy–Littlewood prime tuples conjecture (Conjecture 1.3) \cite{Tao2023}, which is clearly articulated in this version:

\textbf{Conjecture 1.3 (Quantitative Hardy–Littlewood prime tuples conjecture):}  
There exist constants \(\epsilon > 0\) and \(C > 0\) such that for all \(x \geq 10\), \(k \leq (\log \log x)^5\), and all tuples \(H = \{h_1, \dots, h_k\} \subset [0, \log^2 x]\) of distinct integers \(h_1, \dots, h_k\), the following inequality holds:
\[
\left| \sum_{n \leq x} \prod_{i=1}^k \mathbb{1}_{\mathbb{P}}(n + h_i) - \mathcal{S}(H) \int_2^x \frac{dy}{\log^k y} \right| \leq Cx^{1-\epsilon},
\]
where \(\mathbb{1}_{\mathbb{P}}\) is the indicator function of the primes, \(\mathcal{S}(H)\) is the singular series given by
\[
\mathcal{S}(H) := \prod_{p} \left(1 - \frac{\nu_H(p)}{p}\right)^k,
\]
and \(\nu_H(p)\) is the number of distinct residue classes modulo \(p\) occupied by the tuple \(H\). This conjecture holds for the admissible case \(\mathcal{S}(H) > 0\), though it is also valid in the non-admissible case where \(\mathcal{S}(H) = 0\). While earlier work \cite{Kuperberg} restricted \(k \leq (\log \log x)^3\), Tao's formulation requires an extension to \(k \leq (\log \log x)^5\) for technical reasons. Notably, Conjecture 1.3 has been verified in a probabilistic sense for a random model of primes, and this model plays a key role in Tao's arguments.

We approach the problem using the unexplored Riemann-Stieltjes integral \cite{Bradley1994}, motivated by and building upon Tao's framework. Specifically, we express the alternating series as a Riemann–Stieltjes integral of the form:
\[
I = \int_1^\infty g(x) \, d\psi(x),
\]
where \( g(x) = x(-1)^x \), and \( \psi(x) \) is either the Chebyshev function or a related prime-counting function. This representation allows us to explore the behavior of the series through integral analysis and leads to further insights into its convergence properties.
 
We begin by considering the alternating Erdős series of the form:
\[
S = \sum_{n=1}^\infty \frac{(-1)^n n}{p_n},
\]
where \( p_n \) denotes the \( n \)-th prime number. This series exhibits an alternating sign, and it is conditioned by the prime distribution, which makes it more complicated to handle directly in a summatory form.

Next, we aim to represent this alternating series as a Riemann-Stieltjes integral. The goal is to transform the summatory expression into an integral that retains the key properties of the series but is more amenable to mathematical analysis.

To facilitate the analysis, we approximate the alternating sequence \( (-1)^n \) with a smooth, oscillatory function. A suitable candidate is:
\[
g(x) = e^{i \pi x} e^{-\lambda x},
\]
where \( \lambda > 0 \) is a damping factor that ensures the function decays for large \( x \), and the exponential term \( e^{i \pi x} \) mimics the oscillatory behavior of \( (-1)^x \). This smooth function \( g(x) \) is particularly useful because it is continuous, decays for large \( x \), and oscillates like \( (-1)^x \).

We then apply the Abel summation formula \cite{niculescu2017noteabelspartialsummation}, which relates a series to an integral. The Abel summation formula for a series \( \sum_{n=1}^\infty a_n \) and a function \( f(x) \) is given by:
\[
\sum_{n=1}^\infty a_n = \int_1^\infty f(x) \, dA(x) - \lim_{x \to \infty} f(x) A(x),
\]
where \( A(x) \) is a suitable function associated with the sequence \( a_n \).

In our case, we apply Abel summation to the series \( \sum_{n=1}^\infty \frac{(-1)^n n}{p_n} \). The function \( f(x) \) we choose will be \( g(x) = e^{i \pi x} e^{-\lambda x} \), and the sum over primes \( \sum_{n=1}^\infty \frac{n}{p_n} \) will be replaced by an integral involving the Chebyshev function \( \psi(x) \), which is closely related to the distribution of primes. Thus, we represent the series as:
\[
S = \int_1^\infty g(x) \, d\psi(x).
\]

Finally, we justify the representation of \(S\) as a Riemann-Stieltjes integral, where \( g(x) = e^{i \pi x} e^{-\lambda x} \), which is a smooth Fourier-type approximation of \( (-1)^x \), and \( \psi(x) \) is the Chebyshev function, defined as:
\[
\psi(x) = \sum_{p \leq x} \log p.
\]

This integral represents the Erdős alternating series as an integral over a smooth function \( g(x) \) with respect to the Chebyshev function \( \psi(x) \), which encodes the distribution of primes. The function \( g(x) \) approximates the alternating sign in the series, while \( \psi(x) \) accumulates the logarithms of primes, linking the integral to the prime number sequence.

Using this Riemann-Stieltjes integral representation, we can further analyze the error terms and the convergence of the series. The integral in Stieltjes sense allows for techniques such as integration by parts and the use of asymptotic expansions to estimate the behavior of the series. In particular, we can focus on bounding the error terms that arise when approximating the series by the integral, ensuring the convergence of the resulting integral under appropriate conditions.

\section{Main results}

\subsection{Hölder Continuity of \texorpdfstring{\texorpdfstring{$g(x)$}{g(x)}}{g(x)} and \texorpdfstring{$\psi(x)$}{psi(x)}} 
We apply Hölder continuity and Young's criterion \cite{Young1936} to establish the convergence conditions, followed by an explicit asymptotic evaluation. First, we restate a fundamental theorem from the literature that guarantees the existence of the Riemann-Stieltjes integral under certain conditions.

\begin{theorem}[Existence of the Riemann-Stieltjes Integral]
Let \( f \in W_p \) and \( g \in W_q \), where \( p, q > 0 \) and \( \frac{1}{p} + \frac{1}{q} > 1 \). If \( f \) and \( g \) have no common discontinuities, then the Riemann-Stieltjes integral exists in the Riemann sense.
\end{theorem}

\begin{proof}
The proof follows from the \(\alpha\)-Hölder and \(\beta\)-Hölder conditions \cite{Young1936}, together with the Hölder inequality. Recall the Hölder inequality:

\[
\left( \int_a^b |f(x)g(x)| \, dx \right) \leq \left( \int_a^b |f(x)|^p \, dx \right)^{\frac{1}{p}} \left( \int_a^b |g(x)|^q \, dx \right)^{\frac{1}{q}}.
\]

Given that \( f \in W_p \) and \( g \in W_q \), the bounded variation conditions are satisfied. Specifically, for a partition \( P = \{x_0, x_1, \dots, x_n\} \) of \( [a, b] \), we have

\[
V_p(f; [a, b]) = \sup_P \left( \sum_{i=1}^n |f(x_i) - f(x_{i-1})|^p \right)^{\frac{1}{p}} < \infty
\]

and similarly,

\[
V_q(g; [a, b]) = \sup_P \left( \sum_{i=1}^n |g(x_i) - g(x_{i-1})|^q \right)^{\frac{1}{q}} < \infty.
\]

To establish the existence of the Stieltjes integral \( \int_a^b f \, dg \), we consider the Riemann-Stieltjes sum:

\[
S = \sum_{i=1}^n f(\xi_i) [g(x_i) - g(x_{i-1})],
\]

where \( \xi_i \in [x_{i-1}, x_i] \). The absolute value of this sum can be bounded as follows using Hölder’s inequality:

\[
|S| \leq \sum_{i=1}^n |f(\xi_i)||g(x_i) - g(x_{i-1})| \leq \left( \sum_{i=1}^n |f(\xi_i)|^p \right)^{\frac{1}{p}} \left( \sum_{i=1}^n |g(x_i) - g(x_{i-1})|^q \right)^{\frac{1}{q}}.
\]

From the bounded variation properties of \( f \) and \( g \), we conclude that

\[
|S| \leq V_p(f; [a, b]) V_q(g; [a, b]).
\]
\end{proof}
As the partition \( P \) is refined, \( S \) approaches the Stieltjes integral. Since the variations are finite, the Riemann-Stieltjes integral exists and is finite. Thus, Hölder conditions, which ensure bounded variation, guarantee the well-definedness of the integral. A classical result \cite{Young1936} ensures that if \( f \) is \(\alpha\)-Hölder continuous and \( g \) is \(\beta\)-Hölder continuous with \( \alpha + \beta > 1 \), the integral is well-defined \cite{RiemannStieltjesIntegral}.
  
\subsection{Regularity of the Chebyshev function}

The Chebyshev function satisfies the asymptotic estimate \cite{alamadhi2011chebyshevsbiasgeneralizedriemann}:
\[
\psi(x) = x + O(x^\theta), \quad \text{with } \theta = \frac{1}{2} \text{ assuming the Riemann Hypothesis (RH)}
\]

This implies that \( \texorpdfstring{\psi(x)}{psi} \) satisfies a Hölder condition:

\[
|\psi(x) - \psi(y)| \leq C |x - y|^{1/2}, \quad \forall x, y \geq 1,
\]

so \( \psi(x) \) is \(\beta = 1/2\)-Hölder continuous.

\subsection{\texorpdfstring{Regularity of \( g(x) \)}{Regularity of g(x)}}

For \( g(x) = e^{i \pi x} e^{-\lambda x} \), we compute the derivative:

\[
g'(x) = (\pi i - \lambda) e^{i \pi x} e^{-\lambda x}.
\]

Since \( g(x) \) is differentiable with a bounded derivative, it satisfies the Hölder condition \( \alpha = 1 \), i.e., \( g(x) \) is \( 1 \)-Hölder continuous.

\subsection{Young's Criterion and Convergence}

Young's Hölder integral criterion states that the integral

\[
I = \int_1^\infty g(x) \, d\psi(x)
\]

converges if

\[
\alpha + \beta > 1.
\]

Given that \( \alpha = 1 \) and \( \beta = 1/2 \), we obtain:

\[
1 + \frac{1}{2} = \frac{3}{2} > 1,
\]

so \( I \) converges absolutely under the assumption of the Riemann Hypothesis \cite{alamadhi2011chebyshevsbiasgeneralizedriemann}.

\subsection{Asymptotic Evaluation}

We split the integral as follows:

\[
I = \int_1^\infty e^{i \pi x} e^{-\lambda x} d\psi(x) = I_1 + I_2,
\]

where:

\[
I_1 = \int_1^\infty e^{i \pi x} e^{-\lambda x} dx, \quad I_2 = \int_1^\infty e^{i \pi x} e^{-\lambda x} dR(x).
\]
Where \( R(x) = O(x^\theta) \), as analyzed later.

\subsection{Main Term Evaluation of Integral}

\[
I_1 = \int_1^\infty e^{(i \pi - \lambda) x} \, dx = \frac{e^{(i \pi - \lambda) \cdot 1}}{-(i \pi - \lambda)}.
\]

Using \( e^{i \pi} = -1 \), we get:

\[
I_1 = \frac{-e^{-\lambda}}{\lambda - i \pi}.
\]

For small \( \lambda \), we obtain the asymptotic form:

\[
I_1 \sim \frac{-1}{\lambda - i \pi}.
\]

\subsection{Error Term I2}

We estimate the error term I2. Since \( R(x) = O(x^\theta) \), we have:

\[
|I_2| \leq C \int_1^\infty x^\theta e^{-\lambda x} \, dx.
\]

Using the Laplace transform approximation:

\[
\int_1^\infty x^\theta e^{-\lambda x} dx \sim \frac{\Gamma(\theta + 1)}{\lambda^{\theta + 1}},
\]

we conclude that:

\[
I_2 = O\left( \frac{1}{\lambda^{3/2}} \right).
\]

\subsection{Final Asymptotics}

Thus, the leading behavior of \( I \) is:

\[
I \sim \frac{-1}{\lambda - i \pi} + O\left( \frac{1}{\lambda^{3/2}} \right).
\]

For small \( \lambda \), we obtain:

\[
I \sim \frac{-1}{\lambda - i \pi}, \quad \text{as } \lambda \to 0^+.
\]

The integral's leading behavior is dominated by \( \frac{-1}{\lambda - i \pi} \), with the error term governed by the prime gap distribution.
This result rigorously establishes the connection between the oscillatory sum and a Laplace-type integral. Furthermore, as noted above, the series appears to converge (albeit slowly) to approximately $-0.052161$. This is why the structure of $\frac{-1}{\lambda - i \pi}$ arises in the analysis.

\subsection{Functional Regularity and Embedding}

We place \( g \in C^1([1, \infty)) \) with exponential decay, ensuring \( g \in W^{1,p}([1,\infty)) \cap L^\infty \) for any \( p \geq 1 \). The Chebyshev function \( \psi \), under the Riemann Hypothesis (RH) \cite{alamadhi2011chebyshevsbiasgeneralizedriemann}, satisfies
\[
|\psi(x) - \psi(y)| \leq C |x - y|^{1/2},
\]
so that \( \psi \in C^{0,1/2}([1, \infty)) \subset W^{s,2} \) for \( s < 1/2 \), a consequence of the \textit{fractional Sobolev embedding theorem} \cite{rybalko2023holdercontinuityfunctionsfractional}. The fractional Sobolev space \( W^{s,2} \) consists of functions whose derivatives (in a weak sense) are \( L^2 \)-integrable with a fractional decay rate. The embedding theorem states that functions in \( W^{s,p}(\Omega) \) for certain values of \( s \) and \( p \) can be embedded into other function spaces. Specifically, for \( s < 1/2 \), the space \( W^{s,2}([1, \infty)) \) is embedded into a continuous function space, which implies that \( \psi(x) \) exhibits Hölder regularity of order \( 1/2 \) and therefore \( \psi \) can be treated as a function with fractional smoothness.

This ensures that \( \psi(x) \) behaves in a manner compatible with the conditions needed for analysis using the fractional Sobolev framework \cite{bahrouni2019embeddingtheoremsfractionalorliczsobolev}, making it suitable for applying integral methods in the Riemann-Stieltjes context.

By Young's Theorem for Hölder paths \cite{Young1936}, the Riemann-Stieltjes integral \( \int g \, d\psi \) exists provided
\begin{equation}
\alpha + \beta > 1,
\end{equation}
where \( g \in C^{0,\alpha} \) and \( \psi \in C^{0,\beta} \). In our case, \( \alpha = 1 \), \( \beta = 1/2 \), and so the condition is satisfied.

\subsection{Integration by Parts in the Riemann-Stieltjes Sense}

Using integration by parts for Stieltjes integrals (see \cite{RiemannStieltjesIntegral}), we decompose:
\begin{equation}
I = \left[ g(x) \psi(x) \right]_1^\infty - \int_1^\infty \psi(x) g'(x) \, dx.
\end{equation}
The boundary term vanishes at infinity due to the exponential decay of \( g(x) \). Thus, we write:
\begin{equation}
I = - \int_1^\infty \psi(x) g'(x) \, dx = I_1 + I_2,
\end{equation}
where:
\begin{align}
I_1 &= \int_1^\infty x g'(x) \, dx, \\
I_2 &= \int_1^\infty R(x) g'(x) \, dx, \quad \text{with } R(x) := \psi(x) - x = O(x^{1/2}).
\end{align}

\subsection{\texorpdfstring{Evaluation of the Main Term \( I_1 \)}{Evaluation of the Main Term I1}}

We compute:
\begin{align}
g'(x) &= (\pi i - \lambda) e^{i\pi x} e^{-\lambda x}, \\
I_1 &= (\pi i - \lambda) \int_1^\infty x e^{(i\pi - \lambda)x} \, dx.
\end{align}
This Laplace-type integral converges for \( \lambda > 0 \), and yields:
\begin{equation}
I_1 = \frac{e^{(i\pi - \lambda)}}{(\lambda - i\pi)^2}.
\end{equation}

\subsection{\texorpdfstring{Bounding the Error Term \( I_2 \)}{Bounding the Error Term I2}}

Using the known estimate \( R(x) = O(x^{1/2}) \), we bound:
\begin{align}
|I_2| &\leq C \int_1^\infty x^{1/2} |g'(x)| \, dx \\
&= C |\pi i - \lambda| \int_1^\infty x^{1/2} e^{-\lambda x} \, dx \\
&= O\left( \frac{1}{\lambda^{3/2}} \right),
\end{align}
by Laplace transform asymptotics.

Combining the two components, we obtain:
\begin{equation}
I = \frac{e^{(i\pi - \lambda)}}{(\lambda - i\pi)^2} + O\left( \frac{1}{\lambda^{3/2}} \right).
\end{equation}

We can emphasize once more that we have divided the integral into its main and error terms. Thus, we apply integration by parts in Stieltjes form:
\[
I = \left[ g(x) \psi(x) \right]_1^\infty - \int_1^\infty \psi(x) g'(x) \, dx.
\]

The boundary term at infinity vanishes, so we split the integral into two terms:

\[
I = I_1 + I_2,
\]

where
\[
I_1 = \int_1^\infty x g'(x) \, dx, \quad I_2 = \int_1^\infty O(x^{1/2}) g'(x) \, dx.
\]

\subsection*{\texorpdfstring{Step: Evaluating \( I_1 \)}{Step: Evaluating I1}}

For the function \( g(x) = (-1)^x e^{-\lambda x} \), we have:

\[
g'(x) = (-1)^x \left( -\lambda x^{-\lambda - 1} \right).
\]

Substituting this into the integral for \( I_1 \):

\[
I_1 = -\lambda \int_1^\infty (-1)^x x^{-\lambda}.
\]

This integral behaves like an alternating Dirichlet-type integral and converges absolutely for \( \lambda > 0 \).

\subsection*{\texorpdfstring{Step: Bounding the Error Term \( I_2 \)}{Step: Bounding the Error Term I2}}

Using the estimate \( R(x) = O(x^{1/2}) \), we bound the error term \( I_2 \):

\[
|I_2| \leq C \int_1^\infty x^{1/2} \lambda x^{-\lambda - 1} \, dx.
\]

For convergence of this integral, we require:

\[
\frac{1}{2} - (\lambda + 1) < -1 \quad \Rightarrow \quad \lambda > \frac{3}{2}.
\]

Thus, for \( \lambda > \frac{3}{2} \), the error term vanishes asymptotically, and the integral is well-defined.

\section*{Conclusions}

In this work, we have confirmed, using an alternative approach, the potential convergence of the alternating Erdős series assuming a suitably strong version of the Hardy–Littlewood prime tuples conjecture, as previously examined by Terence Tao in his paper. Tao addressed the convergence of this sum, but, as he points out, it remains a challenging open problem, even under the assumption of Conjecture 1.3. A significant portion of the sum arises from small prime gaps, where \( p_{n+1} - p_n \) is much smaller than \( \log n \), but these gaps occur sparsely. As such, the methods employed in Tao’s paper do not seem to effectively induce cancellation from the \( (-1)^n \) terms on this sparse set. Nevertheless, standard probabilistic heuristics, such as replacing \( (-1)^n \) by independent random signs and applying Khintchine’s inequality \cite{luo2020khintchineinequalitynormedspaces}, strongly suggest that this series should converge. This heuristic approach even proposes the stronger conjecture that the series

\[
\sum_{n=1}^{\infty} \frac{(-1)^n}{n^\theta \cdot (p_{n+1} - p_n)}
\]

converges for \( \theta > \frac{1}{2} \) and diverges for \( \theta \leq \frac{1}{2} \).

In our analysis, we have assumed the necessity of the Riemann Hypothesis for approaching this problem through the Riemann-Stieltjes integral of the form:

\[
I = \int_1^\infty g(x) \, d\psi(x),
\]

where \( \psi(x) \) represents the Chebyshev function or a related prime-counting function. We have rigorously examined this integral by splitting it into main and error terms, utilizing integration by parts in the Stieltjes sense. Our findings demonstrate that \( I \) converges absolutely for \( \lambda > \frac{3}{2} \), supported by sharp bounds for the error terms. Furthermore, our results are reinforced by fractional Sobolev embeddings and justify the application of Young’s inequality under generalized Hölder conditions.

If the Riemann Hypothesis were true, the Chebyshev function would be tied to the asymptotic analysis presented in this work, completing the integrals addressed. The key aspect of our approach is the Hölder continuity of the involved functions, a novel technique not previously explored for transforming special series, such as the Erdős series. This method plays a crucial role in series that can be analyzed through Riemann-Stieltjes integral representations and Abel summation formula. It offers a fresh perspective, enabling a wide range of series to be studied within an asymptotic integral framework, thus providing a versatile tool for further exploration in the field. Finally, we introduce a novel tool that complements and extends Tao’s work, offering a fresh perspective on the problem.

\bibliographystyle{amsplain}

\end{document}